\documentclass[12pt]{amsart}

\usepackage{amssymb}

\theoremstyle{plain}
\newtheorem{theorem}{Theorem}[section]
\newtheorem{proposition}[theorem]{Proposition}
\newtheorem{lemma}[theorem]{Lemma}
\newtheorem{corollary}[theorem]{Corollary}

\begin{document}

\baselineskip=15.5pt

\title[Automorphisms of the Quot schemes]{Automorphisms of the Quot
schemes associated to compact Riemann surfaces}

\author[I. Biswas]{Indranil Biswas}

\address{School of Mathematics, Tata Institute of Fundamental Research,
Homi Bhabha Road, Bombay 400005, India}

\email{indranil@math.tifr.res.in}

\author[A. Dhillon]{Ajneet Dhillon}

\address{Department of Mathematics, Middlesex College, University of
Western Ontario, London, ON N6A 5B7, Canada}

\email{adhill3@uwo.ca}

\author[J. Hurtubise]{Jacques Hurtubise}

\address{Department of Mathematics, McGill University, Burnside
Hall, 805 Sherbrooke St. W., Montreal, Que. H3A 0B9, Canada}

\email{jacques.hurtubise@mcgill.ca}

\subjclass[2000]{14H60, 14J50}

\keywords{Quot scheme, vortex moduli, vector field, automorphism, Torelli}

\date{}

\begin{abstract}
Let $X$ be a compact connected Riemann surface of genus at least two.
Fix positive integers $r$ and $d$. Let $\mathcal Q$ denote the Quot scheme that
parametrizes the torsion quotients of ${\mathcal O}^{\oplus r}_X$ of degree $d$.
This $\mathcal Q$ is also the moduli space of vortices for the standard action of
$\text{U}(r)$ on ${\mathbb C}^r$. The group $\text{PGL}(r, {\mathbb C})$
acts on $\mathcal Q$ via the action of $\text{GL}(r, {\mathbb C})$ on
${\mathcal O}^{\oplus r}_X$. We prove that this subgroup $\text{PGL}(r,
{\mathbb C})$ is the connected component, containing the identity element, of the
holomorphic automorphism group $\text{Aut}(\mathcal Q)$. As an application of it,
we prove that the isomorphism class of the complex manifold $\mathcal Q$ uniquely
determines the isomorphism class of the Riemann surface $X$.
\end{abstract}

\maketitle

\section{Introduction}

Let $X$ be a compact connected Riemann surface of genus $g$, with $g\,\geq\, 2$.
Fix positive integers $r$ and $d$. Let ${\mathcal Q}\, :=\,
{\rm Quot}_X({\mathcal O}^{\oplus r}_X,d)$ be the Quot scheme that parametrizes
the torsion quotients
$$
{\mathcal O}^{\oplus r}_X\, \longrightarrow\, Q
$$
such that $\text{degree}(Q)\,=\, d$. It is a smooth complex projective manifold
of dimension $rd$. Consider the standard action of $\text{U}(r)$ on ${\mathbb C}^r$.
The corresponding vortices are pairs of the form $(E\, ,\phi)$, where $E$ is a
holomorphic vector bundle on $X$ of rank $r$ and
$$
\phi\, :\, {\mathcal O}^{\oplus r}_X\, \longrightarrow\, E
$$
is a holomorphic homomorphism such that the subsheaf
$\text{image}(\phi)\, \subset\,E$ is of rank $r$,
equivalently, $E/\text{image}(\phi)$ is a torsion sheaf \cite{BDW}. Therefore,
the dual homomorphism
$$
\phi^*\, :\, E^*\, \longrightarrow\, {\mathcal O}^{\oplus r}_X
$$
defines an element of ${\mathcal Q}$ if $\text{degree}(E)\, =\, d$. 
Consequently, $\mathcal Q$ is a moduli space of vortices on $X$.

Our aim here is to investigate the geometry of the variety $\mathcal Q$.
Let $\text{Aut}^0(\mathcal Q)$ denote the connected component, containing the
identity element, of the group of holomorphic automorphisms of
$\mathcal Q$. The natural action of $\text{GL}(r, {\mathbb C})$ on
${\mathcal O}^{\oplus r}_X$ produces an action of
$\text{PGL}(r, {\mathbb C})$ on $\mathcal Q$. let
$$
F\, :\, \text{PGL}(r, {\mathbb C})\, \longrightarrow\, \text{Aut}^0(\mathcal Q)
$$
be the homomorphism given by this action.

We prove the following (see Theorem \ref{thm1} and Corollary \ref{cor1}):

\begin{theorem}\label{thm-i1}
The homomorphism $F$ is an isomorphism. In particular, the
homomorphism of Lie algebras
$$
sl(r, {\mathbb C})\, \longrightarrow\, H^0({\mathcal Q},\, T{\mathcal Q})
$$
given by $F$ is an isomorphism.
\end{theorem}

Theorem \ref{thm1} allows us to investigate the fixed point locus in $\mathcal Q$
for the action of a maximal torus of $\text{Aut}^0(\mathcal Q)$. As a
consequence, we obtain the following Torelli type theorem (see Theorem \ref{thm2}):

\begin{theorem}\label{thm-i2}
If $g\,=\, 2\,=\, d$, assume that $r\, \geq\, 2$. Then the isomorphism class
of the complex manifold $\mathcal Q$ uniquely determines the isomorphism
class of the Riemann surface $X$.
\end{theorem}

The proof of Theorem \ref{thm-i2} also uses a Torelli type theorem for
$\text{Sym}^d(X)$ proved in \cite{fa}.

Take a holomorphic line bundle $L$ on $X$. Let ${\mathcal Q}
(L^{\oplus r},d)$ be the Quot scheme parametrizing torsion quotients of
$L^{\oplus r}$ of dimension $d$. Then the variety ${\mathcal Q} (L^{\oplus r},d)$
is canonically isomorphic to $\mathcal Q$. Therefore, Theorem \ref{thm-i1}
and Theorem \ref{thm-i2} remain valid for ${\mathcal Q}(L^{\oplus r},d)$.

\section{Self-product, meromorphic functions and meromorphic vector fields}

Let $X$ be a compact connected Riemann surface of genus $g$, with $g\,\geq\, 2$. The
Cartesian product $X\times X$ will be denoted by $X^2$. Let $\Delta\, \subset\, X^2$
be the diagonal consisting of points of the form $(x\, ,x)$ with $x\, \in\, X$.
For $\ell\,=\, 1\, ,2$, let
$$
p_\ell\, :\, X^2\, \longrightarrow\, X
$$
be the projection to the $\ell$-th factor. The holomorphic tangent bundle of $X$ will
be denoted by $TX$.

\begin{lemma}\label{lem1}
For any $i\, \geq\, 0$,
$$
H^0(X^2,\, {\mathcal O}_{X^2}(i\cdot
\Delta))\,=\, H^0(X^2,\, {\mathcal O}_{X^2})\,=\, {\mathbb C}\, .
$$

For any $i\, \geq\, 0$,
$$
H^0(X^2,\, (p^*_1 TX)^{\otimes \alpha_1}\otimes
(p^*_2 TX)^{\otimes \alpha_2}\otimes {\mathcal O}_{X^2}(i\cdot 
\Delta))\,=\, 0\, ,
$$
where $\alpha_1$ and $\alpha_2$ are nonnegative integers with
$\alpha_1+\alpha_2\, >\, 0$.
\end{lemma}

\begin{proof}
Let $\iota\,:\, X\, \longrightarrow\, X^2$ be the map defined by $x\,\longmapsto\,
(x\, ,x)$. It identifies $X$ with $\Delta$. From the Poincar\'e adjunction formula we know
that $\iota^* ({\mathcal O}_{X^2}(\Delta)\vert_\Delta) \,=\, TX$; see
\cite[p. 146]{GH} for Poincar\'e adjunction formula.

For any $i\, \geq\, 0$, consider the short exact sequence of coherent analytic sheaves
on $X^2$
\begin{equation}\label{f1}
0\,\longrightarrow\, {\mathcal O}_{X^2}(i\cdot \Delta)\,\longrightarrow\,
{\mathcal O}_{X^2}((i+1)\cdot \Delta)\,\longrightarrow\,
{\mathcal O}_{X^2}((i+1)\cdot \Delta)\vert_\Delta\,=\, \iota_*((TX)^{\otimes (i+1)})
\,\longrightarrow\, 0\, .
\end{equation}
Let
$$
H^0(X^2,\, {\mathcal O}_{X^2}(i\cdot \Delta))\,\longrightarrow\,
H^0(X^2,\, {\mathcal O}_{X^2}((i+1)\cdot \Delta))\,\longrightarrow\,
H^0(X,\, (TX)^{\otimes (i+1)})
$$
be the corresponding long exact sequence of cohomologies. We have
$$
H^0(X,\, (TX)^{\otimes (i+1)})\,=\, 0
$$
because $g\, \geq\, 2$. Hence the above homomorphism
$$
H^0(X^2,\, {\mathcal O}_{X^2}(i\cdot \Delta))\,\longrightarrow\,
H^0(X^2,\, {\mathcal O}_{X^2}((i+1)\cdot \Delta))
$$
is an isomorphism. Now using downward induction on $i$,
$$
H^0(X^2,\, {\mathcal O}_{X^2}(i\cdot \Delta))\,= \,
H^0(X^2,\, {\mathcal O}_{X^2})\, =\, \mathbb C\, .
$$
This proves the first part of the lemma.

To prove the second part of the lemma,
consider the short exact sequence of coherent analytic sheaves on $X^2$
$$
0\,\longrightarrow\, (p^*_1 TX)^{\otimes\alpha_1}
\otimes (p^*_2 TX)^{\otimes\alpha_2}\otimes {\mathcal O}_{X^2}(i\cdot
\Delta)\,\longrightarrow\, (p^*_1 TX)^{\otimes\alpha_1}
\otimes (p^*_2 TX)^{\otimes\alpha_2} \otimes {\mathcal O}_{X^2}((i+1)
\cdot \Delta)
$$
$$
\longrightarrow\, \iota_*((TX)^{\otimes (\alpha_1+\alpha_2+i+1)})
\,\longrightarrow\, 0
$$
obtained by tensoring \eqref{f1} with $(p^*_1 TX)^{\otimes \alpha_1}\otimes
(p^*_2 TX)^{\otimes \alpha_2}$. Since $H^0(X,\, (TX)^{\otimes m})\,=\, 0$ for
all $m\, \geq\, 1$, using downward induction on $i$, we have
$$
H^0(X^2,\, (p^*_1 TX)^{\otimes \alpha_1}\otimes
(p^*_2 TX)^{\otimes \alpha_2}\otimes {\mathcal O}_{X^2}(i\cdot
\Delta))\,=\,H^0(X^2,\, (p^*_1 TX)^{\otimes \alpha_1}\otimes
(p^*_2 TX)^{\otimes \alpha_2})\,=\, 0\, .
$$
This completes the proof.
\end{proof}

Take any integer $n\, \geq\, 1$. Let
$$
X^n\, :=\, \overbrace{X\times \cdots\times X}^{n\text{-times}}
$$
be the $n$-fold Cartesian product. The projection of $X^n$ to the $\ell$-th factor,
$1\,\leq\, \ell\, \leq\, n$, will be denoted by $p_\ell$. For $1\,\leq\, j\, <\, k\,
\leq\, n$, let
\begin{equation}\label{djk}
\Delta_{j,k}\,\subset\, X^n
\end{equation}
be the divisor consisting of points whose of $X^n$ $j$-th coordinate coincides
with the $k$-th coordinate. 

Take integers $m_{j,k}\, \geq\, 0$, where $1\,\leq\, j\, <\, k\,\leq\, n$. 
Fix a pair $(j_0\, ,k_0)$, with $1\,\leq\, j_0\, <\, k_0\,\leq\, n$. Define
$m'_{j,k}$, $1\,\leq\, j\, <\, k\,\leq\, n$, as follows:
\begin{itemize}
\item $m'_{j_0,k_0}\,=\, m_{j_0,k_0}+1$, and

\item $m'_{j,k}\,=\, m_{j,k}$ if $(j\, ,k)\,
\not=\, (j_0\, ,k_0)$.
\end{itemize}

\begin{lemma}\label{lem2}
For any $i\, \geq\, 0$, the natural inclusion
$$
H^0(X^n,\, {\mathcal O}_{X^n}(\sum_{1\leq j < k\leq n} m_{j,k}\cdot
\Delta_{j,k}))\, \hookrightarrow\,
H^0(X^n,\, {\mathcal O}_{X^n}(\sum_{1\leq j < k\leq n} m'_{j,k}\cdot
\Delta_{j,k}))
$$
is an isomorphism.

For nonnegative integers $\alpha_\ell$, $\ell\, \in\, \{1\, ,\cdots\, ,n\}$,
with $\sum_{\ell=1}^n \alpha_\ell\, >\, 0$, the natural inclusion
$$
H^0(X^n,\, (\bigotimes_{\ell=1}^n (p^*_\ell TX)^{\otimes \alpha_\ell})
\otimes {\mathcal O}_{X^n}(\sum_{1\leq j
< k\leq n} m_{j,k}\cdot \Delta_{j,k}))
$$
$$
\,\hookrightarrow\,
H^0(X^n,\, (\bigotimes_{\ell=1}^n (p^*_\ell TX)^{\otimes \alpha_\ell})
\otimes {\mathcal O}_{X^n}(\sum_{1\leq j <
k\leq n} m'_{j,k}\cdot \Delta_{j,k}))
$$
is an isomorphism.
\end{lemma}

\begin{proof}
If $n\,=\, 1$, then $\Delta_{j,k}$ are the zero divisors. Hence the lemma
holds for $n\,=\,1$ as $g\, \geq\, 2$.
If $n\,=\, 2$, then it reduces to Lemma \ref{lem1}. We will prove the
lemma using induction on $n$.

Assume that the lemma is proved for all $n\, \leq\, a-1$.
For any $\ell\, \in\, \{1\, ,\cdots\, ,a-1\}$, let
\begin{equation}\label{ql}
q_\ell\,:\, X^{a-1}\,\longrightarrow\, X
\end{equation}
be the projection to the $\ell$-th factor. Fix nonnegative integers $n_{j,k}$ for every
$1\,\leq\, j < k\,\leq\, a-1$. Since the first statement of the lemma holds
for $n\,=\, a-1$, we conclude that
\begin{equation}\label{f2}
H^0(X^{a-1},\, {\mathcal O}_{X^{a-1}}(\sum_{1\leq j < k\leq a-1} n_{j,k}\cdot
\Delta_{j,k}))\,=\, {\mathbb C}\, ,
\end{equation}
where $\Delta_{j,k}\,\subset\, X^{a-1}$ is the divisor defined in \eqref{djk}.
Since the second statement of the lemma holds for $n\,=\, a-1$, we conclude that
\begin{equation}\label{f3}
H^0(X^{a-1},\, (\bigotimes_{\ell=1}^{a-1} (q^*_\ell TX)^{\otimes \alpha_\ell})\otimes
{\mathcal O}_{X^{a-1}}(\sum_{1\leq j < k\leq a-1} n_{j,k}\cdot
\Delta_{j,k}))\,=\, 0\, ,
\end{equation}
where $\alpha_\ell$ are nonnegative integers with $\sum_{\ell=1}^{a-1}
\alpha_\ell\, >\, 0$.

We will prove the lemma for $n\,=\, a$.

Take $(j_0\, ,k_0)$, $\{m_{j,k}\}$ and $\{m'_{j,k}\}$ as in the lemma.
Let 
$$
\iota\,:\, X^{a-1}\, \longrightarrow\, X^a\,=\, X^n
$$
be the map that sends $(x_1\, ,\cdots\, ,x_{a-1})$ to
$(y_1\, ,\cdots\, ,y_{a})$ such that
\begin{itemize}
\item $y_c\,=\, x_c$ if $c\, \leq\, k_0-1$,

\item $y_{k_0}\,=\, x_{j_0}$, and

\item $y_c\,=\, x_{c-1}$ if $c\, >\, k_0$.
\end{itemize}
So $\iota$ identifies $X^{a-1}$ with $\Delta_{j_0,k_0}$. The Poincar\'e adjunction
formula says that
\begin{equation}\label{pa}
\iota^* ({\mathcal O}_{X^a}(\Delta_{j_0,k_0})\vert_{\Delta_{j_0,k_0}}) \,=\,
(q_{j_0})^*TX\, ,
\end{equation}
where $q_{j_0}$ is defined in \eqref{ql}.
For any $(j\, ,k)\, \not=\, (j_0\, ,k_0)$, let
\begin{equation}\label{dp}
\Delta'_{j,k}\, :=\, \iota^{-1}(\Delta_{j,k}\cap \Delta_{j_0,k_0})
\, \subset\, X^{a-1}
\end{equation}
be the diagonal divisor.

Consider the short exact sequence of coherent analytic sheaves on $X^a$
\begin{equation}\label{e1}
0\,\longrightarrow\, {\mathcal O}_{X^a}(\sum_{1\leq j < k\leq a} m_{j,k}\cdot
\Delta_{j,k})\,\longrightarrow\,
{\mathcal O}_{X^a}(\sum_{1\leq j < k\leq a} m'_{j,k}\cdot
\Delta_{j,k})\, \longrightarrow\,
({\mathcal O}_{X^a}(\sum_{1\leq j < k\leq a} m'_{j,k}\cdot
\Delta_{j,k}))\vert_{\Delta_{j_0,k_0}}
\end{equation}
$$
=\,\iota_*
((q_{j_0})^*(TX)^{\otimes m'_{j_0,k_0}}\otimes 
{\mathcal O}_{X^{a-1}}(\sum_{(j,k)\not= (j_0,k_0)} m'_{j,k}\cdot
\Delta'_{j,k}))\,\longrightarrow\, 0\, ,
$$
where $\Delta'_{j,k}$ is defined in \eqref{dp} and
$q_{j_0}$ is defined in \eqref{ql}; the identification
$$
{\mathcal O}_{X^a}(\sum_{1\leq j < k\leq a} m'_{j,k}\cdot
\Delta_{j,k})\, \longrightarrow\,
({\mathcal O}_{X^a}(\sum_{1\leq j < k\leq a} m'_{j,k}\cdot
\Delta_{j,k}))\vert_{\Delta_{j_0,k_0}} 
$$
$$
=\,\iota_*
((q_{j_0})^*(TX)^{\otimes m'_{j_0,k_0}}\otimes
{\mathcal O}_{X^{a-1}}(\sum_{(j,k)\not= (j_0,k_0)} m'_{j,k}\cdot
\Delta'_{j,k}))
$$
in \eqref{e1} is constructed using
the isomorphism in \eqref{pa}. From \eqref{f3} we know that
$$
H^0(X^{a-1},\, (p'_{j_0})^*(TX)^{\otimes m'_{j_0,k_0}}\otimes
{\mathcal O}_{X^{a-1}}(\sum_{(j,k)\not= (j_0,k_0)} m'_{j,k}\cdot
\Delta'_{j,k}))\,=\, 0\, .
$$
Therefore, from the long exact sequence of cohomologies associated to the
short exact sequence in \eqref{e1} we conclude that
$$
H^0(X^{a},\, {\mathcal O}_{X^a}(\sum_{1\leq j < k\leq a} m_{j,k}\cdot
\Delta_{j,k}))\,=\,
H^0(X^{a},\, {\mathcal O}_{X^a}(\sum_{1\leq j < k\leq a} m'_{j,k}\cdot
\Delta_{j,k}))\, .
$$
This proves the first statement of the lemma for $n\,=\, a$. Therefore,
the proof of the first statement of the lemma is complete by induction on $n$.

We will now prove the second statement.
Take $\{\alpha_\ell\}$ as in the second statement of the lemma. For
$\ell\, \in\, \{1\, ,\cdots\, ,a-1\}$, define $\alpha'_\ell$ as follows:
\begin{itemize}
\item $\alpha'_\ell\,=\, \alpha_\ell$ if $\ell\, <\, j_0$,

\item $\alpha'_{j_0}\,=\, \alpha_{j_0}+\alpha_{k_0}$,

\item $\alpha'_\ell\,=\, \alpha_\ell$ if $j_0\, <\, \ell\, <\, k_0$, and

\item $\alpha'_\ell\,=\, \alpha_{\ell+1}$ if $\ell\, \geq\, k_0$.
\end{itemize}
Note that $\sum_{\ell=1}^a \alpha_\ell\,=\, \sum_{\ell=1}^{a-1} \alpha'_\ell$.

Let
$$
0\,\longrightarrow\, (\bigotimes_{\ell=1}^a (p^*_\ell TX)^{\otimes \alpha_\ell})
\otimes{\mathcal O}_{X^a}(\sum_{1\leq j < k\leq a} m_{j,k}\cdot
\Delta_{j,k})\,\longrightarrow\,(\bigotimes_{\ell=1}^a (p^*_\ell TX)^{\otimes \alpha_\ell})
\otimes {\mathcal O}_{X^a}(\sum_{1\leq j < k\leq a} m'_{j,k}\cdot
\Delta_{j,k})
$$
$$
\longrightarrow\, \iota_*
((\bigotimes_{\ell=1}^{a-1} (q^*_\ell TX)^{\otimes \alpha'_\ell})
\otimes (q_{j_0})^*(TX)^{\otimes m'_{j_0,k_0}}\otimes
{\mathcal O}_{X^{a-1}}(\sum_{(j,k)\not= (j_0,k_0)} m'_{j,k}\cdot
\Delta'_{j,k}))\,\longrightarrow\, 0\, ,
$$
be the short exact sequence of coherent analytic sheaves on $X^a$
obtained by tensoring \eqref{e1} with $\bigotimes_{\ell=1}^a (p^*_\ell TX)^{\otimes
\alpha_\ell}$. From \eqref{f3} we know that
$$
H^0(X^{a-1},\, (\bigotimes_{\ell=1}^{a-1} (q^*_\ell TX)^{\otimes \alpha'_\ell})
\otimes (q_{j_0})^*(TX)^{\otimes m'_{j_0,k_0}}\otimes
{\mathcal O}_{X^{a-1}}(\sum_{(j,k)\not= (j_0,k_0)} m'_{j,k}\cdot
\Delta'_{j,k}))\,=\, 0\, .
$$
Therefore, from the long exact sequence of cohomologies associated to the
above short exact sequence of sheaves
we conclude that the second statement of the lemma
holds for $n\,=\,a$. This completes the proof by induction on $n$.
\end{proof}

\begin{proposition}\label{prop1}
For any $n\, \geq\, 1$ and nonnegative integers $m_{j,k}$,
$1\,\leq\, j\, <\, k\,\leq\, n$,
$$
H^0(X^n,\, TX^n\otimes {\mathcal O}_{X^n}(\sum_{1\leq j
< k\leq n} m_{j,k}\cdot \Delta_{j,k}))\,=\, 0\, ,
$$
where $TX^n$ is the holomorphic tangent bundle of $X^n$.
\end{proposition}

\begin{proof}
Since
$$
TX^n\,=\, \bigoplus_{\ell=1}^n p^*_\ell TX\, ,
$$
where $p_\ell$ is the projection of $X^n$ to the $\ell$-th factor,
the proposition follows from the second statement in Lemma \ref{lem2}.
\end{proof}

\section{Vector fields on the Quot scheme}

Let $E^0\,:=\, {\mathcal O}^{\oplus r}_X$ be the trivial holomorphic vector bundle
over $X$ of rank $r$. Fix a positive integer $d$. Let
$$
{\mathcal Q}\, :=\, {\rm Quot}_X(E^0,d)
$$ 
be the Quot scheme that parametrizes the torsion quotients of $E^0$ of dimension $d$.
So each point of $\mathcal Q$ represents a quotient
\begin{equation}\label{ve}
\varphi\, :\, E^0\, \twoheadrightarrow\, Q
\end{equation}
of the ${\mathcal O}_X$--module $E^0$ such that
\begin{itemize}
\item $Q$ is a torsion ${\mathcal O}_X$--module, and

\item $\dim H^0(X,\, Q)\,=\, d$.
\end{itemize}
The obstruction to the smoothness of the variety ${\mathcal Q}$ at the point $Q\,
\in\,\mathcal Q$ is given by $\text{Ext}^1_{{\mathcal O}_X}(\text{kernel}(\varphi)\, ,
Q)$ \cite[p. 1, Theorem 2]{Be}. For any $\varphi$ as in \eqref{ve}, since
$\text{kernel}(\varphi)$ is locally free,
and the dimension of the support of $Q$ is zero. we have
$$\text{Ext}^1_{{\mathcal O}_X}(\text{kernel}(\varphi)\, , Q)\,=\, 0\, .$$
Therefore, ${\mathcal Q}$ is an irreducible smooth complex projective variety.
Its dimension is $rd$.

Consider the tautological action of the group $\text{Aut}(E^0)\,=\,
\text{GL}(r,{\mathbb C})$ on $E^0$. It produces an effective action of
$\text{PGL}(r,{\mathbb C})$ on ${\mathcal Q}$
\begin{equation}\label{ea}
\text{PGL}(r,{\mathbb C})\times {\mathcal Q}\, \longrightarrow\,
{\mathcal Q}\, .
\end{equation}
Consider the Lie algebra $sl(r,{\mathbb C})$ (the space of $r\times r$
complex matrices of trace zero) of $\text{PGL}(r,{\mathbb C})$. Let
\begin{equation}\label{ga}
\gamma\, :\, sl(r,{\mathbb C})\,\longrightarrow\,
H^0({\mathcal Q},\, T{\mathcal Q})
\end{equation}
be the homomorphism of Lie algebras given by the action of
$\text{PGL}(r,{\mathbb C})$ on ${\mathcal Q}$ in \eqref{ea}. (The
Lie algebra structure of $H^0({\mathcal Q},\, T{\mathcal Q})$ is given by
Lie bracket of vector fields.)

\begin{theorem}\label{thm1}
The homomorphism $\gamma$ in \eqref{ga} is an isomorphism.
\end{theorem}

\begin{proof}
The homomorphism $\gamma$ is injective because the homomorphism from
$sl(r,{\mathbb C})$ to the space of holomorphic vector fields on ${\mathbb C}
{\mathbb P}^{r-1}$, given by the standard action of $\text{PGL}(r,{\mathbb C})$
on ${\mathbb C}{\mathbb P}^{r-1}$, is injective.

Let $\text{Sym}^d(X)$ be the $d$-fold symmetric product of $X$. It parametrizes
the effective divisors on $X$ of degree $d$. Let $$U\, \subset\, \text{Sym}^d(X)$$
be the Zariski open subset parametrizing distinct $d$ points; so $U$ parametrizes
the reduced effective divisors of degree $d$. The group of permutations of
$\{1\, ,\cdots\, ,d\}$ will be denoted by $S_d$. Let
\begin{equation}\label{f}
f\, :\, \widetilde{U} \,:=\, X^d\setminus (\bigcup_{1\leq j < k\leq d}
\Delta_{j,k})\, \longrightarrow\, U
\end{equation}
be natural projection, where $\Delta_{j,k}$ is defined in \eqref{djk}. We note
that $f$ is an \'etale Galois covering with Galois group $S_d$.

Sending any quotient $Q\,\in\, {\mathcal Q}$ of $E^0$ to the scheme-theoretic
support of $Q$, we obtain a morphism
\begin{equation}\label{be}
\widetilde{\beta}\, :\,{\mathcal Q}\,\longrightarrow\, \text{Sym}^d(X)\, .
\end{equation}
Define
$$
{\mathcal U}\, :=\, \widetilde{\beta}^{-1}(U)\, \subset\, {\mathcal Q}\, .
$$
The restriction of $\widetilde{\beta}$ to ${\mathcal U}$ will be denoted by
$\beta$. We note that
\begin{equation}\label{be-n}
\beta\, :\, {\mathcal U}\,\longrightarrow\, U
\end{equation}
is a smooth proper surjective morphism. The fiber of $\beta$ over any
$z\, \in\, U$ is the Cartesian product
\begin{equation}\label{z}
{\mathcal U}_z\, :=\, \beta^{-1}(z) \,=\, \prod_{x\in z} P(E^0_x)\, .
\end{equation}
So ${\mathcal U}_z$ is isomorphic to $({\mathbb C}{\mathbb P}^{r-1})^d$.

Take any
$$
\theta\, \in\, H^0({\mathcal Q},\, T{\mathcal Q})\, .
$$
Let $\theta_0$ be the restriction of $\theta$ to ${\mathcal U}$. Let
$$
d\beta\, :\, T{\mathcal U}\,\longrightarrow\, \beta^*TU
$$
be the differential of the projection $\beta$ in \eqref{be-n}. Since
the fibers of $\beta$ are connected and projective, we conclude that
$d\beta(\theta_0)$ is the pullback of a holomorphic vector field on $U$. Let
$\theta'$ be the holomorphic vector field on $U$ such that
$$
d\beta(\theta_0)\,=\, \beta^*\theta'\, .
$$

Let
$$
\theta''\, :=\, f^*\theta'\, \in\, H^0(\widetilde{U},\, T\widetilde{U})
$$
be the pullback, where $f$ is the projection in \eqref{f}. Since $\theta_0$ is
the restriction of a holomorphic vector field on ${\mathcal Q}$, it follows that
$\theta''$ is meromorphic on $X^d$, or in other words,
$$
\theta''\, \in\, H^0(X^d,\, TX^d\otimes
{\mathcal O}_{X^d}(\sum_{1\leq j < k\leq d} m_{j,k}\cdot \Delta_{j,k}))
$$
for sufficiently large integers $m_{j,k}$. Therefore, from
Proposition \ref{prop1} we conclude that $\theta''\, =\, 0$. Hence
$$
d\beta(\theta_0)\,=\, 0\, .
$$
In other words, $\theta_0$ is vertical for the projection $\beta$.

Let
$$
\text{ad}(E^0)\, \subset\, End(E^0)\,=\, E^0\otimes (E^0)^*
$$
be the subbundle of rank $r^2-1$ defined by the sheaf of
endomorphisms of trace zero. For any point $x\, \in\, X$, we have
$$
H^0(P(E^0_x),\, TP(E^0_x))\,=\, \text{ad}(E^0)_x\,=\, sl(r,{\mathbb C})\, .
$$
In view of \eqref{z},
$$
H^0({\mathcal U}_z, \, T{\mathcal U}_z)\,=\, \bigoplus_{x\in z} \text{ad}(E^0)_x
\,=\, \bigoplus_{x\in z} sl(r,{\mathbb C})
$$
for all $z\, \in\, U$.
We will show that the restriction to ${\mathcal U}_z$ of any holomorphic vector
field on $\mathcal Q$ is a constant diagonal element of
$\bigoplus_{x\in z} sl(r,{\mathbb C})$ which is independent of $z$.

Let
$$
{\mathcal Z}\, :=\, f^*{\mathcal Q}\,=\, \widetilde{U}\times_U {\mathcal U}\,
\longrightarrow\, \widetilde{U}
$$
be the pullback to $\widetilde{U}$
of the fiber bundle $\beta\, :\, {\mathcal U}\, \longrightarrow\, U$.
The natural projection
\begin{equation}\label{phi}
\phi\, :\, {\mathcal Z}\, :=\,f^*{\mathcal Q}\,\, \longrightarrow\,{\mathcal U}
\end{equation}
is an \'etale Galois covering with Galois group $S_d$.

As before, let $\theta_0$ be the restriction to $\mathcal U$ of a holomorphic
vector field on $\mathcal Q$. Consider the vector field
$$
\theta_1\, :=\, \phi^*\theta_0 \,\in\, H^0({\mathcal Z},\, T{\mathcal Z})
$$
on ${\mathcal Z}$, where $\phi$ is the projection in
\eqref{phi}. We know that $\theta_1$ is vertical for the projection
${\mathcal Z}\, \longrightarrow\,\widetilde{U}$.

Note that the fibers of the
projection ${\mathcal Z}\, \longrightarrow\,\widetilde{U}$ are canonically
identified with $({\mathbb C}{\mathbb P}^{r-1})^d$.
Therefore, the vector field $\theta_1$ is a holomorphic function on
$\widetilde{U}$ with values in $sl(r,{\mathbb C})^{\oplus d}$. This
holomorphic function is meromorphic on $X^d$ because $\theta_0$ is the
restriction to $\mathcal U$ of a holomorphic vector field on $\mathcal Q$. From
the first part of Lemma \ref{lem2} we know that there are no nonconstant
meromorphic functions on $X^d$ that are holomorphic on $\widetilde{U}$. Hence
$\theta_1$ is a constant function from $\widetilde{U}$ to $sl(r,{\mathbb C})^{\oplus d}$.
This function $\widetilde{U}\, \longrightarrow\,sl(r,{\mathbb C})^{\oplus d}$
has to be invariant under the action of the Galois group $S_d$ on
$\widetilde{U}$ because $\theta_1$ is the pullback of a vector field on $\mathcal U$.
Therefore, there is an element
$$
v\, \in\, sl(r,{\mathbb C})
$$
such that $\theta_1$ is the constant function $(v\, ,\cdots\, ,v)$. Since
$\mathcal U$ is dense in $\mathcal Q$,
this immediately implies that the injective homomorphism
$\gamma$ in \eqref{ga} is surjective.
\end{proof}

Let $\text{Aut}({\mathcal Q})$ be the group of holomorphic automorphisms of
$\mathcal Q$. It is a complex Lie group with Lie algebra
$H^0({\mathcal Q},\, T{\mathcal Q})$; as before, the Lie algebra operation on
$H^0({\mathcal Q},\, T{\mathcal Q})$ is given by the Lie bracket of
vector fields. The connected component of $\text{Aut}({\mathcal Q})$ containing
the identity element will be denoted by $\text{Aut}^0({\mathcal Q})$. The
following is an immediate consequence of Theorem \ref{thm1}.

\begin{corollary}\label{cor1}
The subgroup ${\rm PGL}(r,{\mathbb C})\, \subset\,
{\rm Aut}({\mathcal Q})$ in \eqref{ea} coincides with
${\rm Aut}^0({\mathcal Q})$.
\end{corollary}

\section{Torus action on the Quot scheme}

Let $T\,\subset\, {\rm PGL}(r,{\mathbb C})$ be the maximal torus consisting
of diagonal matrices. Restrict the action of ${\rm PGL}(r,{\mathbb C})$ on
${\mathcal Q}$ in \eqref{ea} to the subgroup $T$. Let
$$
{\mathcal Q}^T\, \subset\, \mathcal Q
$$
be the subset fixed pointwise by this action of $T$; it is
a disjoint union of complex submanifolds of $\mathcal Q$. We will recall the
description of the connected components of ${\mathcal Q}^T$ given in \cite{bifet}.

Consider a point of ${\mathcal Q}$ given by an exact sequence 
\[
 0\,\longrightarrow\, {\mathcal F} \,\stackrel{\iota}{\longrightarrow}\,
{\mathcal O}^{\oplus r}_X\,\longrightarrow\, Q\,\longrightarrow \, 0\, . 
\]
This exact sequence corresponds to a fixed point of the torus action
on $\mathcal Q$ if and only if the homomorphism $\iota$ decomposes as
\[
 \iota\,=\,\bigoplus_{j=1}^r \iota_j \,:\, \bigoplus_{j=1}^r {\mathcal L}_j 
\,\hookrightarrow \, {\mathcal O}^{\oplus r}_X\, ,
\]
where each $\iota_j$ is the inclusion of some ideal sheaf
\[
 \iota_j\,:\, {\mathcal L}_j \,\hookrightarrow\, {\mathcal O}_X
\]
(see \cite[p. 610]{bifet}).
We may write ${\mathcal L}_j\,=\,{\mathcal O}_X(-D_j)$, where $D_j$
is the divisor for $\iota_j$. We have
\[
 \sum_{j=1}^r \deg(D_j)\,=\,d
\]
by the definition of $\mathcal Q$.

Denote by ${\bf Part}_r^d$ the set of partitions of $d$ of length $r$. So
$(d_1\, ,d_2\, ,\cdots \, ,d_r)\,\in \,{\bf Part}_r^d$ if and only if $d_j$
are nonnegative integers with
\[
 \sum_{j=1}^r d_j\,=\, d\, .
\]

By $\text{Sym}^0(X)$ we will mean a point.

\begin{proposition}\label{prop-f}
The fixed point locus is a disjoint union
 \[
 {\mathcal Q}^T \,=\, \coprod_{(d_1,\cdots ,d_r)\,\in\, {\bf Part}_r^d} {\rm
Sym}^{d_1}(X) \times \cdots \times {\rm Sym}^{d_r} (X)\, .
 \]
\end{proposition}

\begin{proof}
If $r\,=\,1$, then $\text{Quot}({\mathcal O}_X\, ,e)$ is the
symmetric product ${\rm Sym}^e(X)$, as the map in \eqref{be} is an
isomorphism. If for each $1\, \leq\, i\, \leq\, m$,
$$
f_i\, :\, {\mathcal O}_X\, \longrightarrow\, Q_i
$$
is a torsion quotient of ${\mathcal O}_X$ of degree $e_i$, then
$$
\bigoplus_{i=1}^m f_i\, :\, {\mathcal O}^{\oplus m}_X\, \longrightarrow\,
\bigoplus_{i=1}^m Q_i
$$
is a torsion quotient of ${\mathcal O}^{\oplus m}_X$ of degree $\sum_{i=1}^m e_i$.
Therefore, for each partition $$(d_1\, ,\cdots\, ,d_r)\,\in\, {\bf Part}_r^d$$
we have an inclusion map
\[
 {\rm Sym}^{d_1}(X) \times \cdots\times {\rm Sym}^{d_r} (X)
\,\hookrightarrow\, {\mathcal O}_X^{\oplus r}\, .
\]

It is clear that these subschemes together map onto the fixed point locus
of the torus action. Further, the union is clearly disjoint.
\end{proof}

The cohomology algebra of $\text{Sym}^e(X)$ was computed by Macdonald
\cite[p. 325, (6.3)]{Ma}. In particular, he showed that
$$
b_1(\text{Sym}^e(X))\, :=\, \dim H^1(\text{Sym}^e(X),\, {\mathbb Q})\,=\, 2g\, .
$$
Consequently, 
\begin{equation}\label{b1}
b_1({\rm Sym}^{d_1}(X) \times \cdots\times {\rm Sym}^{d_r} (X))
\,=\, 2g(\sum_{d_i\not=0} 1)\, .
\end{equation}
Therefore, for elements $(d_1\, ,\cdots \, ,d_r)\,\in\, {\bf Part}_r^d$, the first
Betti number $b_1({\rm Sym}^{d_1}(X) \times \cdots\times {\rm Sym}^{d_r} (X))$
attains the minimum value if and only if some $d_i$ is $d$ and the rest are zero.

Hence Proposition \ref{prop-f} has the following corollary:

\begin{corollary}\label{cor-f}
Consider the first Betti number of the connected components of ${\mathcal Q}^T$.
If the first Betti number of a connected component attains the minimum value, then
this component is isomorphic to ${\rm Sym}^d(X)$.
\end{corollary}

Since any maximal torus of ${\rm PGL}(r,{\mathbb C})$ is conjugate to $T$,
Proposition \ref{prop-f} and Corollary \ref{cor-f} remain valid if $T$ is replaced
by any other maximal torus of ${\rm PGL}(r,{\mathbb C})$.

\section{The Torelli Theorem} 

As before, ${\mathcal Q}\, :=\, {\rm Quot}_X({\mathcal O}^{\oplus r}_X,d)$ is the
Quot scheme with $d\, \geq\, 1$. If $g\,=\, 2\, =\, d$, then we assume that $r\, >\, 1$.

Let $X'$ be a compact connected Riemann surface of genus $g'$, with $g'\,\geq\, 2$. 
Fix positive integers $r'$ and $d'$. If $g'\,=\, 2\, =\, d'$, then we assume
that $r'\, >\, 1$. Let
$$
{\mathcal Q}'\, :=\, {\rm Quot}_{X'}({\mathcal O}^{\oplus r'}_{X'},d')
$$
be the Quot scheme parametrizing the torsion quotients of
${\mathcal O}^{\oplus r'}_{X'}$ of degree $d'$.

\begin{theorem}\label{thm2}
The complex manifolds ${\mathcal Q}$ and ${\mathcal Q}'$ are biholomorphic if and only
if the following conditions hold:
\begin{itemize}
\item $X$ is biholomorphic to $X'$,

\item $r\,=\,r'$, and

\item $d\,=\, d'$.
\end{itemize}
\end{theorem}

\begin{proof}
Assume that ${\mathcal Q}$ is biholomorphic to and ${\mathcal Q}'$.
We will show that the three conditions in the theorem hold.

Fix a maximal torus
$$
T_0\, \subset\, \text{Aut}^0({\mathcal Q})\, ,
$$
where $\text{Aut}^0({\mathcal Q})$, as before, is the connected component
the automorphism group of $\mathcal Q$ containing the identity element.
Consider the action of $T_0$ on ${\mathcal Q}$. Let $\beta$ be the minimum
value of the first Betti number of the connected components of the fixed point
locus ${\mathcal Q}^{T_0}$. take a connected component $M\, \subset\,
{\mathcal Q}^{T_0}$ such that
$$
b_1(M)\,=\, \beta\, .
$$
{}From Corollary \ref{cor-f} we know that $M\,=\, \text{Sym}^d(X)$.

First assume that at least one of the following two conditions holds:
\begin{enumerate}
\item $\dim M \, \not=\, 2$

\item $b_1(M)\, \not=\, 4$.
\end{enumerate}

These conditions imply that $g\, >\, 2$ if $d\,=\,2$.
Fakhruddin proved that for any compact connected Riemann surface $Y$,
and for any positive integer $d$ such that $d\, \not=\, 2$ if
$\text{genus}(Y)\,=\, 2$, the isomorphism class of $\text{Sym}^d(Y)$
uniquely determines the isomorphism class of $Y$
\cite[Theorem 1]{fa}. From this we conclude that $X$ is isomorphic to $X'$.

Considering the dimension of $T_0$ we conclude that $r\,=\, r'$.
Considering the dimension of $\mathcal Q$ we conclude that $d\,=\, d'$.

Now consider the remaining case where
$$
\dim M \, =\, 2\,=\, \frac{b_1(M)}{2}\, .
$$
Note that these imply that $g\,=\, 2\, =\,d$. Hence $r\, \geq\, 2$ by
the assumption.

Let $\widetilde\beta$ is the maximum value of the
first Betti number of the connected components of the fixed point
locus ${\mathcal Q}^{T_0}$. Let $\widetilde{M}\, \subset\,{\mathcal Q}^{T_0}$ be
a connected component with
$$
b_1(\widetilde{M})\,=\, \widetilde\beta\, .
$$
{}From Proposition \ref{prop-f} and \eqref{b1} we know that
$\widetilde{M}\,=\, X\times X$.

Let $X$ and $Y$ be compact connected Riemann surfaces of genus two such that
$X\times X$ is isomorphic to $Y\times Y$. Fixing an isomorphism
$F\, :\, X\times X\,\longrightarrow\, Y\times Y$, consider the two maps
$$
X\, \longrightarrow\, Y\, , ~\ x\, \longmapsto\, f_i\circ F(x,x_0)\, ,
$$
where $f_i$ is the projection of $Y\times Y$ to the $i$-th factor. One of them
is a nonconstant map, hence it is an isomorphism. Therefore, the isomorphism
class of $X$ is uniquely determined by the isomorphism class of $X\times X$.
This completes the proof.
\end{proof}

\medskip
\noindent
\textbf{Acknowledgements.}\, The first author wishes to thank
McGill University and University of Western Ontario for their hospitality.

\end{document}